\documentclass[10 pt]{IEEEtran}

\usepackage{graphicx}
\usepackage{amsfonts,amssymb}
\usepackage{amsmath,amsthm}
\usepackage{cases}
\usepackage{arydshln}
\usepackage{mathrsfs}
\usepackage{algorithm}
\usepackage{algorithmic}
\usepackage{geometry}
\usepackage[colorlinks,
linkcolor=blue,
anchorcolor=blue,
citecolor=red
]{hyperref}
\usepackage{cite}
\usepackage{subcaption}
\newtheorem{thm}{Theorem}
\newtheorem{lem}{Lemma}
\newtheorem{rem}{Remark}
\newtheorem{pro}{Proposition}
\newtheorem{coro}{Corollary}

\newtheorem{defn}{Definition}
\newtheorem{exam}{Example}
\newtheorem{assu}{Assumption}

\newcommand{\rank}{\mathrm{Rank}}

\newcommand{\sprank}{\mathrm{Sprank}}

\begin{document}

	\title{A Necessary Condition for Network Identifiability with Partial Excitation and Measurement}
	
	\author{Xiaodong~Cheng, 
			Shengling Shi,
			 Ioannis Lestas, and
			 Paul M.J. Van den Hof 
			 %<-this % stops a space <-this % stops a space
			\thanks{This work is supported by 
				the European Research Council (ERC), Advanced Research Grant SYSDYNET (No. 694504) and Starting Grant HetScaleNet (No. 679774).}%
		   \thanks{
				X. Cheng and I. Lestas are with the Control Group,  Department
				of Engineering, University of Cambridge, CB2 1PZ, United Kingdom.
				{\tt\small \{xc336, icl20\}@cam.ac.uk}}
			\thanks{
				S. Shi and P. M. J. Van den Hof are with the Control Systems Group, Department of Electrical Engineering,
				Eindhoven University of Technology, 
				5600 MB Eindhoven,
				The Netherlands.
				{\tt\small \{s.shi, P.M.J.vandenhof\}@tue.nl}}%

	}
	
	\maketitle

\begin{abstract}
	This paper considers dynamic networks where vertices and edges represent manifest signals and causal dependencies among the signals, respectively. We address the problem of how to determine if the dynamics of a network can be identified when only partial vertices are measured and excited. 
	A necessary condition for network identifiability is presented, where the analysis is performed based on identifying the dependency of a set of rational functions from excited vertices to measured ones. This condition is further characterised by using an edge-removal procedure on the associated bipartite graph. Moreover, on the basis of necessity analysis, we provide a necessary and sufficient condition for identifiability in circular networks. 
\end{abstract}
%\begin{IEEEkeywords}
%	Identifiability, graph theory, system identification,
%	dynamic networks, linear systems.
%\end{IEEEkeywords}

\section{Introduction} 

The study of complex dynamic networks has flourished recently in the systems and control community, along with a wide range of applications in e.g., robotic coordination, biochemical reactions, and smart power grids. Developing mathematical models, i.e. system identification, for these interconnected systems is a fundamental step in understanding their behaviour and eventually devising efficient techniques for prediction and control. 

In this paper, we consider identification of a class of dynamic networks consisting of vertex signals that are interconnected by causal rational transfer functions (modules) and possibly driven by external excitation signals \cite{paul2013netid}.  A central concept here is \textit{identifiability}, which essentially reflects if a unique model can be distinguished on the basis of measurement data. Different from identifiability defined in the classical system identification literature for fixed open-loop and closed-loop configurations \cite{ljung1999system}, identifiability analysis in a network setting largely relies on the interconnection structure of networks \cite{gonccalves2008LTInetworks,chetty2017necessary,adebayo2012dynamical,weerts2015sysid,hayden2016sparse,bazanella2017identifiability,hendrickx2018identifiability,henk2018identifiability}. Based on the structural information, a set of models for a dynamic network is obtained, and then the core problem of this paper is to \textit{explore the conditions under which the network model set is identifiable}.  

The identifiability problem can be formulated in the scale of a full network, see e.g., \cite{bazanella2017identifiability,hendrickx2018identifiability,Bazanella2019CDC,henk2018identifiability,weerts2018identifiability,cheng2019allocation}, and it can also focus on a single module or a subset of modules in a network \cite{gevers2018practical,weerts2018single,shi2020subnetworks,shi2020single}.  
In this paper, the latter is of particular interest. The majority of existing studies on the full network identifiability mainly consider two settings. In e.g., \cite{bazanella2017identifiability,hendrickx2018identifiability,henk2018identifiability}, all vertices are assumed to be excited by external excitation signals, and only partial vertex signals are measured, while the works e.g., \cite{weerts2018identifiability,cheng2019CDC,cheng2019allocation} follow a dual setting, where all internal variables are measured, and only a subset vertices are driven by measured external excitation signals or unmeasured noises. Within these two settings, 
necessary and sufficient conditions for network identifiability have been derived. 
In \cite{weerts2018identifiability}, identifiability is interpreted as the full rank property of certain transfer matrix from external signals to measured internal signals. In contrast, a generic notion of identifiability is proposed in \cite{bazanella2017identifiability,hendrickx2018identifiability}, which leads to attractive graph-theoretical conditions based on \textit{vertex-disjoint paths} for checking network identifiability. These works have inspired \cite{henk2018identifiability}, which extends the path-based condition to address global identifiability. Furthermore, \cite{cheng2019allocation} reformulates a new graphical characterisation for generic identifiability by means of \textit{disjoint pseudotree covering}, which further leads to a scalable graph-based algorithm for allocating actuators.

While the above-mentioned works require that all vertices are either measured or excited by sufficiently rich external signals, \cite{Bazanella2019CDC} provides identifiability results in the scenario where not all vertices are measured and not all vertices are excited. Sufficient conditions are discussed for the generic identifiability of a subset of modules. Yet these conditions require certain prior knowledge on network dynamics, and they are not entirely graph-based except for special networks with cycle and tree topologies. A generic local notion of network identifiability is considered in \cite{legat2020local} that is weaker than generic identifiability. A necessary and sufficient condition for local identifiability can be obtained in terms of transfer matrix ranks, but a graph-theoretical analysis for local identifiability remains an open question. In contrast, we present in \cite{shi2020single} a sufficient condition for generic identifiability  of a single module in a network. 

In line with the network setting of \cite{Bazanella2019CDC,shi2020single}, this paper studies identifiability of full dynamic networks, where only partial excitation and measurement signals are available. 
We develop a new necessary condition for the identifiability of general networks by recasting the identifiability problem into determining the dependency of a system of rational functions with the parametrised modules as indeterminate variables, where the functions represent individual mappings from the excitation signals to the measured vertices. The condition can be reformulated in terms of an edge-removal process for a bipartite graph, thus it only relies on the topology of networks. Furthermore, the necessity analysis is then applied to circular  networks, leading to a necessary and sufficient condition for identifiability that goes beyond the results in \cite{Bazanella2019CDC}, where only a sufficient condition is given.

	The rest of this paper is organized as follows: In Section~\ref{sec:preliminaries}, we recap some basic notations used in graph theory and introduce the dynamic network model. Section \ref{sec:result} presents a necessary condition for the identifiability of dynamic networks and further zooms into the analysis of circular  networks, for which a necessary and sufficient condition for identifiability is provided. Finally, conclusions are made in Section \ref{sec:conclusion}.

\textit{Notation:}
Denote $\mathbb{R}$ as the set of real numbers, and $\mathbb{R}(q)$ is the rational function field over $\mathbb{R}$ with the variable $q$. The cardinality of a set $\mathcal{V}$ is represented by $\lvert \mathcal{V} \rvert$. $A_{ij}$ denotes the $(i,j)$-th entry of a matrix $A$, and more generally, $A_{\mathcal{U},\mathcal{V}}$ denotes the submatrix of $A$ that consists of the rows and columns of $A$ indexed by two positive integer sets $\mathcal{U}$ and $\mathcal{V}$, respectively. The normal rank of a transfer matrix $A(q)$ is denoted by $\rank(A(q))$, and $\rank(A(q)) = r$ if the rank of $A(q)$ is equal to $r$ for almost all values of $q$.

	\section{Preliminaries and Problem Setting}
	\label{sec:preliminaries}
	\subsection{Graph Theory}
	
	A graph $\mathcal{G}$ consists of a finite and nonempty vertex set  $\mathcal{V}: = \{1, 2, \cdots , L\}$ and an edge set $\mathcal{E} \subseteq \mathcal{V} \times \mathcal{V}$.
	A directed graph is such that each element in $\mathcal{E}$ is an ordered pair of elements of $\mathcal{V}$. If $(i,j) \in \mathcal{E}$, we say that the vertex $i$ is an \textit{in-neighbour} of $j$, and $j$ is an \textit{out-neighbour} of $i$. We use $\mathcal{N}_j^-$ and $\mathcal{N}_j^+$ to denote the sets that collect all the in-neighbours and out-neighbours of vertex $j$, respectively.
	
	A graph $\mathcal{G}$ is called \textit{simple}, if $\mathcal{G}$ does not contain self-loops (i.e., $\mathcal{E}$ does not contain any edge of the form $(i,i)$, $\forall~i\in \mathcal{V}$), and there exists only one directed edge from one vertex to its each out-neighbour. In a simple graph, a directed \textit{path} connecting vertices
	$i_0$ and $i_n$ is a sequence of  edges of the form $(i_{k-1}, i_k)$, $k = 1, \cdots, n$, and every vertex appears at most once on the path. 
	Particularly, a single vertex can also be regarded as a special path of length $0$.
	Two directed paths are \textit{vertex-disjoint} if they do not share any common vertex, including the start and the end vertices.  
	%In a simple directed graph $\mathcal{G}$, we denote $b_{\mathcal{U} \rightarrow \mathcal{Y}}$ as the maximum number of mutually vertex-disjoint
	%paths from $\mathcal{U} \subseteq \mathcal{V}$ to $\mathcal{Y} \subseteq \mathcal{V}$. .
	%A directed simple graph $\mathcal{G}$ is \textit{connected} if the underlying undirected graph $\mathcal{G}_\mathrm{u}$ obtained by replacing all directed edges of $\mathcal{G}$ with undirected edges is connected, i.e., in $\mathcal{G}_\mathrm{u}$, there is an undirected path between any pair of vertices.
	
	%A \textit{source} in a simple graph $\mathcal{G}$ is a vertex without any in-neighbors, and likewise, a \textit{sink} is a vertex without any out-neighbors. 
	%The sources and sinks of $\mathcal{G}$ are collected by the following sets, respectively.
	% Let ${\mathcal S_\mathrm{ou}}(\mathcal{G})$ and ${\mathcal{S}_\mathrm{in}}(\mathcal{G})$ be the sets of sources and sinks of $\mathcal{G}$, respectively.

	\subsection{Dynamic Network Model}
	
	Consider a dynamic network whose topology is captured by a simple directed graph $\mathcal{G} = (\mathcal{V}, \mathcal{E})$ with vertex set
	$\mathcal V = \{1, 2, \cdots , L\}$ and edge set $\mathcal E \subseteq \mathcal V \times \mathcal V$. Following the basic setup in  \cite{paul2013netid,Bazanella2019CDC}, each vertex describes by an internal variable $w_j(t) \in \mathbb{R}$, and 
	a compact form of the overall network dynamics is 
	\begin{align} \label{eq:net}
	w(t) & = G(q) w(t) +  R r(t) + v_\mathrm{e}(t), \\
	y(t) & = C w(t) + v_\mathrm{m}(t), \nonumber
	\end{align}
	where $q^{-1}$ is the delay operator, and $w(t): = \begin{bmatrix}
	w_1(t) & w_2(t)& \cdots & w_L(t)
	\end{bmatrix}^\top$ collects all the internal signals. $G(q)$ is a hollow transfer matrix, in which the $(i,j)$-th entry, denoted by $G_{ij}(q) \in \mathbb{R}(q)$, indicates the transfer operator from vertex $j$ to vertex $i$.
	
	Let $\mathcal{R} \subseteq \mathcal{V}$ and $\mathcal{C} \subseteq \mathcal{V}$ be the vertices that are excited and measured, respectively, and $K = |\mathcal{R}|$ and $N = |\mathcal{C}|$. The signals $r(t) \in \mathbb{R}^K$ and $y(t) \in \mathbb{R}^N$ are the external excitation and measurement signals with $R  \in \mathbb{R}^{L \times K}$, $C \in \mathbb{R}^{N \times L}$ binary   matrices indicating which vertices are excited or measured. Specifically, $R$ and $C$  consist of the columns and rows of the identity matrix indexed by the set $\mathcal{R}$ and $\mathcal{C}$, respectively. The excitation and measurement noises are represented by $v_\mathrm{e}(t) \in \mathbb{R}^L$, $v_\mathrm{m}(t)\in \mathbb{R}^L$, respectively. 
	
	\begin{assu} \label{Assum}
		Throughout the paper, we consider the following standard assumptions for dynamic networks (see also  e.g., \cite{weerts2018identifiability,weerts2018single}).
		\begin{enumerate}
%			\item $r(t)$ is persistently exciting, i.e. the spectrum of $r(t)$ contains a sufficiently large number of harmonics \cite{ljung1999system}.
			
			\item The network \eqref{eq:net} is \textit{well-posed} and stable, i.e., $(I - G(q))^{-1}$ is proper
			and stable.
			
			\item The function $G_{ji}(q)$ is nonzero if and only if $(i, j) \in \mathcal{E}$.
			
			\item All the entries of $G(q)$ are proper and stable transfer operators.
			
		\end{enumerate}
	\end{assu}
	
	In the context of network identification, we consider $M = (G, R, C)$ to be a network model of \eqref{eq:net}, where all the nonzero entries in $G$ are parametrized independently. Therefore, we obtain a network model set
	\begin{equation} \label{eq:modelset}
	\mathcal{M}: = \{M(q,\theta) = (G(q,\theta), R, C), \theta \in \Theta\}.
	\end{equation} 
	Denote the transfer matrix
	\begin{equation} \label{eq:T}
	T(q, \theta) : = (I - G(q,\theta))^{-1},
	\end{equation} 
	and $T_{\mathcal{C},\mathcal{R}}$ is the submatrix of $T$ containing the rows and columns of $T$ indexed by $\mathcal{C}$ and $\mathcal{R}$, respectively.
	The network identifiability is thereby defined as follows.
	\begin{defn}[Network identifiability]
		\label{defn:netid} 
		The network model set $\mathcal{M}$ in \eqref{eq:modelset} is identifiable from the submatrix $T_{\mathcal{C},\mathcal{R}}$ at $M_0: = M(\theta_0)$ with $\theta_0 \in \Theta$
		if the
		implication
		\begin{align}\label{eq:implication}
		C T(q, \theta_1) R =  C T(q, \theta_0) R
		\Rightarrow
		M(q,\theta_1) = M(q,\theta_0),
		\end{align}
		holds for all parameters $\theta_1 \in \Theta$. Furthermore, the network model set $\mathcal{M}$ is identifiable from  $T_{\mathcal{C},\mathcal{R}}$ if \eqref{eq:implication}
		holds for all $\theta_0 \in \Theta$.
	\end{defn}

	As a relevant concept, generic identifiability of the network model set $\mathcal{M}$ is defined when the
	implication \eqref{eq:implication} holds for almost all $\theta_0 \in \Theta$\footnote{``Almost all'' refers to the exclusion of parameters that are in a subset of $\Theta$ with Lebesgue measure zero.}, see more details in \cite{bazanella2017identifiability,hendrickx2018identifiability,weerts2018single}. 
	\begin{rem}	
%		Note that the variable $\theta \in \Theta$ is only used for formalizing the model set $\mathcal{M}$, while the properties of the mapping from $\theta$ to network models is not the focus of this study. 
		It is worth noting that while this paper considers only excitation input $r$ in \eqref{eq:net}, disturbances can also be taken into account as \cite{cheng2019allocation}. Under some mild assumptions, disturbance inputs play a similar role as the excitation inputs, and thereby the results in this paper can be directly generalised to the disturbance case.
	\end{rem}
	Under Assumption~\ref{Assum}, the measured
	signals $y(t)$ and $r(t)$ lead to unique object $T_{\mathcal{C},\mathcal{R}}(z)$. Therefore, identifiability reflects the ability to distinguish between models in the set $\mathcal{M}$ from measurement data, or more preciously, from the transfer matrix $T_{\mathcal{C},\mathcal{R}}$ as described in Definition~\ref{defn:netid}. In this sense, network identifiability essentially depends on the presence and location of external excitation signals $r$ and the selection of measured vertex signals $y$.

	%\begin{prob}
	%	pro
	%\end{prob}
	\section{Existing Necessary Conditions}
	
%	In this paper, we study necessary conditions for identifiability of model set $\mathcal{M}$ of a dynamic network with partial excitation and measurement. The existing results are presented in this section.

	A primary necessary condition has been given in  \cite{Bazanella2019CDC} as follows.
	\begin{lem} \label{lem:exormeas} \cite{Bazanella2019CDC}
		Any network model set $\mathcal{M}$ is identifiable \textit{only if} $\mathcal{V} = \mathcal{R} \cup \mathcal{C}$.
	\end{lem}
	This condition means that to identify all the parametrised entries in $G$, each vertex in $\mathcal{G}$ is either excited or measured.

	In other studies including \cite{hendrickx2018identifiability,weerts2018identifiability,shi2020IFAC}, necessary conditions for network identifiability have been provided in the setting of either full excitation or full measurement. Combining these works in the two settings, we immediately obtain a necessary condition for identifiability in the case of partial excitation and measurement.
	
	\begin{pro}
		\label{pro:necessary1}
		Consider the network model set $\mathcal{M}$ in \eqref{eq:modelset} with $\mathcal{R} \subseteq \mathcal{V}$ and $\mathcal{C} \subseteq \mathcal{V}$ the excited and measured vertices. If $\mathcal{M}$ is identifiable, then 
		\begin{subequations} \label{eq:rankcond}
			\begin{align}
			\rank \left(T_{\mathcal{N}_i^-,\mathcal{R}}(q, \theta)\right) = |\mathcal{N}_i^-|, 
			\label{eq:rankcond1}
			\\ \text{and} \ 
			\rank \left(T_{\mathcal{C},\mathcal{N}_i^+}(q,\theta)\right) = |\mathcal{N}_i^+|,
			\label{eq:rankcond2}
			\end{align}
		\end{subequations}
		hold for each $i \in \mathcal{V}$ and for all $\theta \in \Theta$.
	\end{pro}
	\begin{proof}
		The necessity of \eqref{eq:rankcond1} can be proved following a reasoning similar to Theorem~2 in  \cite{weerts2018identifiability} for the full measurement case. Then, the necessity of \eqref{eq:rankcond2} is also validated, following from a dual analysis.  
	\end{proof}
	
	%The rank of a transfer matrix can be characterized using a graphical concept, called vertex-disjoint paths \cite{vanderWoude1991graph}.

	%As shown in \cite{hendrickx2018identifiability,henk2018identifiability}, the two rank conditions in \eqref{eq:rankcond} can be verified using attractive graph-based conditions: 
	%\begin{coro} \label{coro:pathcond}
	%	If a network model set $\mathcal{M}$ in \eqref{eq:modelset} is generically  (globally) identifiable, then for each vertex $i \in \mathcal{V}$, 
	%	\begin{enumerate}
	%		\item the maximum number of (constrained) vertex disjoint paths from excited vertices $\mathcal{R}$ to $\mathcal{N}_i^-$ should equal to $|\mathcal{N}_i^-|$, and simultaneously, 
	%		\item  the maximum number of (constrained) vertex disjoint paths from $\mathcal{N}_i^+$ to the measured vertices $\mathcal{C}$ should equal to $|\mathcal{N}_i^+|$.
	%	\end{enumerate}
	%\end{coro}
	%The above path-based conditions can be easily checked on the underlying topology of a dynamic network, for instance, the network examples in Fig.~\ref{fig:orginalnet} and Fig.~\ref{fig:transposenet}. Besides, in the case that all the vertices are either excited or measured, the conditions in \eqref{eq:rankcond} become sufficient for network identifiability, as shown in \cite{hendrickx2018identifiability,weerts2018identifiability}. 
	
	Generally, the study of the necessary condition for network identifiability in the partial excitation and measurement setting is rarely addressed.
	The available necessary conditions can be rather loose in determining identifiability of general networks. For instance, these conditions cannot even be used to check identifiability of the four-vertex dynamic network in Fig.~\ref{fig:fournodes}.
	
	\begin{exam}
		\label{ex:4nodes}
		In the network shown in Fig.~\ref{fig:fournodes}, $\mathcal{R} = \{1, 2\}$ and $\mathcal{C} = \{3, 4\}$. The matrix $T$  defined in \eqref{eq:T} is computed as
		\begin{equation*}
		T = \begin{bmatrix}
		1 & 0 & 0 & 0 \\
		G_{21} & 1 & 0 &0\\
		G_{31} & 0 & 1  &0\\
		G_{21}G_{42} + G_{31}G_{43} & G_{42} & G_{43} & 1
		\end{bmatrix}.
		\end{equation*} 
		It is not hard to verify that all the necessary conditions in Lemma~\ref{lem:exormeas} and Proposition~\ref{pro:necessary1} are fulfilled. However, the model set of this network $\mathcal{M}$ is not identifiable, which can simply be seen from the submatrix
		\begin{equation*}
		T_{\mathcal{C},\mathcal{R}} = \begin{bmatrix}
		G_{31} & 0 \\
		G_{21}G_{42} + G_{31}G_{43} & G_{42} 
		\end{bmatrix}.
		\end{equation*}
		Note that identifiability of $\mathcal{M}$ essentially requires to obtain a unique solution of the four unknown  modules $G_{12}$, $G_{24}$, $G_{31}$, and $G_{43}$ from the entries of $T_{\mathcal{C},\mathcal{R}}$. However,  $T_{\mathcal{C},\mathcal{R}}$ has a zero entry, thus it is impossible to identify four modules from it.
		
		\begin{figure}[!tp]
			\centering
			\includegraphics[width=0.2\textwidth]{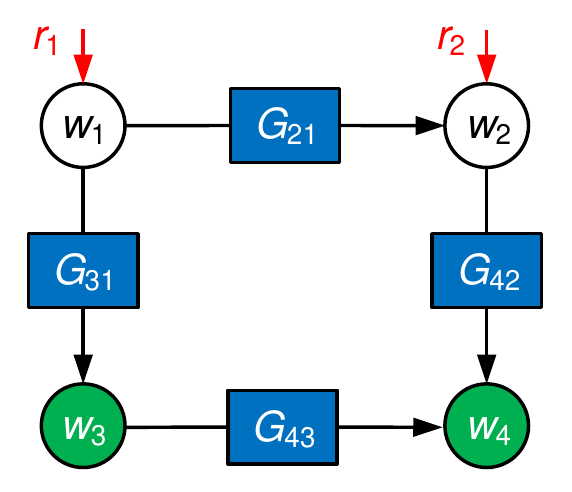}
			\caption{A dynamic network that satisfies the necessary conditions in Lemma~\ref{lem:exormeas} and Proposition~\ref{pro:necessary1} but is not identifiable.}
			\label{fig:fournodes}
		\end{figure}%
	\end{exam}

	\section{Main Results}
	\label{sec:result}
	In this section, we present a novel necessary condition for network identifiability and further perform the analysis to circular networks, yielding a necessary and sufficient condition.
	
	\subsection{A Necessary Condition for Network Identifiability}

	In this section, we derive a necessary condition for network identifiability, motivated by Example~\ref{ex:4nodes}, that is based on the number of the nonzero transfer functions as the entries in $T_{\mathcal{C},\mathcal{R}}$. In the transfer matrix $T$, we have three categories of entries, namely , `0', `1', and \textit{nonconstant elements} that are represented as functions of modules $G_{ji}$, e.g., the element $T_{41}$ in Example~\ref{ex:4nodes} is a function of the modules as $T_{41} = G_{21}G_{42} + G_{31}G_{43}$. For any $i \neq j$, the entry $T_{ji}$ is nonzero if there is at least a directed path from $i$ to $j$, and $T_{ji} = 0$ otherwise. Furthermore, a diagonal entry of $T$, $T_{ii} \ne 1$ if there exists a directed cycle that starts and ends at vertex $i$, and $T_{ii} = 1$ otherwise.

		Identifiability essentially reflects whether %$T_{\mathcal{C},\mathcal{R}}$ 
	%	leads to a unique network model, or equivalently, whether 
		we can uniquely solve all the modules in $G$ from the entries of $T_{\mathcal{C},\mathcal{R}}$. Note that the $0$ or $1$ entries in $T$ do not contain any information of the modules and thus are not useful for identifiability. Let $\xi$ be the number of nonconstant elements in $T_{\mathcal{C},\mathcal{R}}$. A immediate necessary condition for identifiability of  is 
	\begin{equation} \label{eq:naive}
	\xi \geq |\mathcal{E}|,
	\end{equation}
	where $\mathcal{E}$ is the edge set of the network. 
	%\begin{pro}\label{lem:nainec}
	%	Consider the network model set $\mathcal{M}$ in \eqref{eq:modelset}. Let $\xi$ be the number of nonconstant elements in $T_{\mathcal{C},\mathcal{R}}$. If $\mathcal{M}$ is identifiable, then \begin{equation} \label{eq:naive}
	%		\xi \geq |\mathcal{E}|,
	%	\end{equation}
	%	where $\mathcal{E}$ is the edge set of the network. 
	%\end{pro}
	%\begin{proof}
	%	The proof can be shown by reflecting the identifiability problem to the solution of a system of nonlinear equations:
	%	\begin{equation} \label{eq:sysnonlineqs}
	%		\begin{cases}
	%		f_1 (x_1,...,x_{|\mathcal{E}|}) = t_1, \\
	%		f_2 (x_1,...,x_{|\mathcal{E}|}) = t_2, \\
	%		\quad \vdots \\
	%		f_{\xi} (x_1,...,x_{|\mathcal{E}|}) = t_{\xi}, \\
	%		\end{cases}
	%	\end{equation}
	%	where $x_1, ..., x_{|\mathcal{E}|}$ are scalar indeterminate variables corresponding to the $|\mathcal{E}|$ modules in $G$, $t_1, ..., t_{\xi}$ correspond the nonconstant elements in $T_{\mathcal{C},\mathcal{R}}$ that are  identified from the measured data $(r,y)$, and $f_1, ..., f_{\xi}$ are rational functions describing how the nonconstant elements in $T_{\mathcal{C},\mathcal{R}}$ are represented by the $|\mathcal{E}|$ modules in $G$.  
	%	Clearly, it is necessary to have \eqref{eq:naive}  to yield a unique solution of the variables.
	%\end{proof}
	This simple condition can determine identifiability of the network in Fig.~\ref{fig:fournodes}. It is found that $\xi = 3$, while there are $4$ edges, and therefore this does not conform with \eqref{eq:naive}. Consequently, the network model set in Example~\ref{ex:4nodes} is not identifiable. However, the condition in \eqref{eq:naive} is obviously not enough to verify the identifiability of a dynamic network in a more general setting. For instance, it is incapable to handle the network in Fig.~\ref{fig:8nodes}. This network example will also be used as a lead-in to our new necessary condition to be developed later in this section.

	\begin{exam}
		\label{ex:8nodes}
		We consider a simple network in Fig.~\ref{fig:8nodes}, where $\mathcal{R} = \{1, 2, 3, 4\}$ and $\mathcal{C} = \{5, 6, 7, 8\}$.   
%		How to determine identifiability of the model set for this simple network? 
		To the best of our knowledge, there is currently no available condition in the literature that  can  determine identifiability of the model set.
		It is clear that this network satisfies all the available necessary conditions in Lemma~\ref{lem:exormeas} and Proposition~\ref{pro:necessary1}. Furthermore, there are $10$ nonconstant elements that is equal to the number of unknown modules, i.e. $|\mathcal{E}| = 10$. Thus, \eqref{eq:naive} is also satisfied. 
		
		Here, we need a further analysis to determine whether this network model set is identifiable or not. 
		Note that we can identify the submatrix of $T: = (I - G)^{-1}$ from measurement data $(r, y)$ as
		\begin{equation*}
		T_{\mathcal{C},\mathcal{R}} =  \begin{bmatrix}
		T_{51} &  0     & 0      & 0    \\
		T_{61} & T_{62} & 0      & 0    \\
		T_{71} & T_{72} & T_{73} & 0    \\
		T_{81} & T_{82} & T_{83} & T_{84}
		\end{bmatrix},
		\end{equation*}
		where 
		$T_{51} = G_{51}$, $T_{62} = G_{62}$, $T_{73} = G_{73}$, $T_{84} = G_{84}$,
		$T_{61} = G_{62} G_{21} + G_{65} G_{51}$,
		$T_{72} = G_{73} G_{32} + G_{76} G_{62}$,
		$T_{83} = G_{84} G_{43} + G_{87} G_{73}$,
		$T_{71} = G_{73} G_{32} G_{21} + G_{76}  G_{62} G_{21} + G_{76} G_{65} G_{51}$,
		$T_{82} = G_{84} G_{43} G_{32} + G_{87} G_{73} G_{32} + G_{87} G_{76} G_{62}$,
		$T_{81} = G_{84} G_{43} G_{32} G_{21} + G_{87} G_{73} G_{32} G_{21} + G_{87} G_{76} G_{62} G_{21}  + G_{87} G_{76} G_{65} G_{51}
		$. 
		
		It can be verified that
		\begin{align} \label{ex:8nodes_det}
		T_{83} = T_{73}(T_{61} T_{82} - T_{62} T_{81}) (T_{61} T_{72} - T_{62} T_{71})^{-1},
		\end{align}
		which implies that the information of $T_{83}$ is redundant in identifying the unknown modules in the network, since it can be represented by the other nonzero elements in $T_{{\mathcal{C}}
			,{\mathcal{R}}}$. 
		We thus have 10 equations that are not independent, from which it is impossible to solve 10 unknown modules. As a result, the network model set is not identifiable.
		
		\begin{figure}[!tp]
			\centering
			\includegraphics[width=0.38\textwidth]{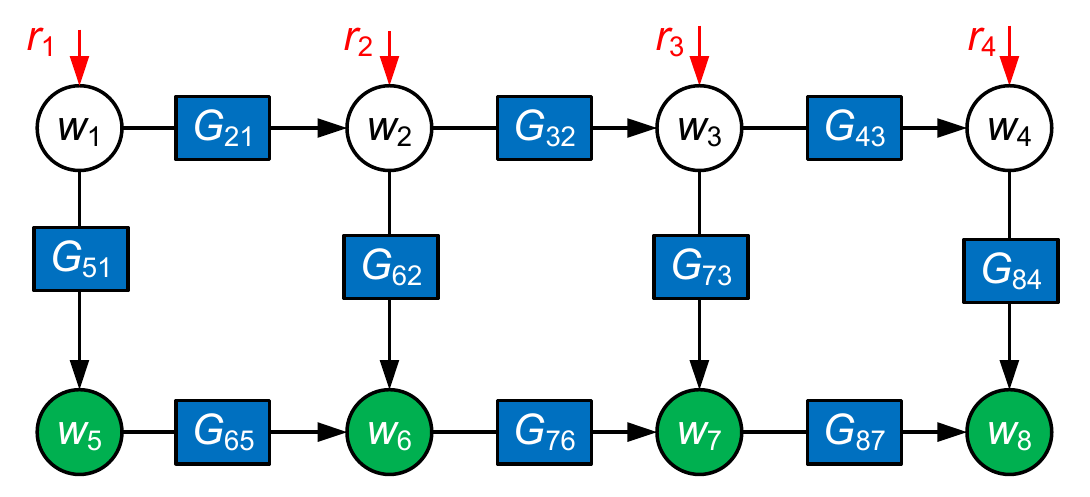}
			\caption{A network that satisfies the necessary conditions in Lemma~\ref{lem:exormeas}, Proposition~\ref{pro:necessary1} and \eqref{eq:naive} but is not identifiable.}
			\label{fig:8nodes}
		\end{figure}%

	\end{exam}
	
	Motivated by Example~\ref{ex:8nodes}, we aim to find a tighter necessary condition for network identifiability, of which the necessity is featured by the independence of a set of rational functions. To this end, we define two sets associated with a dynamic network \eqref{eq:net}. The unknown modules $G_{ij}$ in the network are viewed as indeterminate variables, leading to a set of unknown modules:
	$$
	\mathscr{X} : = \{G_{ji}(q) \mid (i,j) \in \mathcal{E} \},
	$$
	with $\mathcal{E}$ the edge set of the network. Then the nonconstant elements of $T_{\mathcal{C},\mathcal{R}}$ form a set of rational functions on $\mathscr{X}$, 
	denoted by 
	\begin{equation} \label{eq:Fset}
	\mathscr{F}  = \{T_{\ell k } \mid T_{\ell k } \neq 0~\text{and}~T_{\ell k } \neq 1, \ell\in \mathcal{C}, k \in \mathcal{R}\},
	\end{equation}
	where $|\mathscr{F}| = \xi$ with $\xi$ in \eqref{eq:naive}.
	The rational functions in $\mathscr{F}$ are \textit{dependent}, if there is a function in $\mathscr{F}$ that can be represented by the other functions in $\mathscr{F}$ with elementary arithmetic operations (i.e. addition, subtraction, multiplication, and division).

	For instant, the 10 functions in Example~\ref{ex:8nodes} are dependent due to  \eqref{ex:8nodes_det}. In order to analyse identifiability of the network, we can first remove dependent functions in $\mathscr{F}$, e.g., $T_{83}$ in Example~\ref{ex:8nodes}, to obtain a reduced set $\hat{\mathscr{F}}$, and then compare the cardinality of $\hat{\mathscr{F}}$ with the number of edges in the network. A key question therefore arises, that is how to \textit{identify the dependency of the functions in $\mathscr{F}$}?

	To address this question, we adopt the concept of \textit{structural matrix} from e.g., \cite{Steffen2005Structural} to capture the nonzero pattern of $T$.  
	\begin{defn} [Structural matrix and structural rank]
		\label{defn:strucmat}
		The nonzero pattern of $T$ in \eqref{eq:T} can be characterised by a so-called structural matrix ${S}$ defined as
		\begin{equation} \label{eq:strucmat}
		S_{ij} =	\begin{cases}
		0, & \text{if}~T_{ij}(q) = 0,   \\
		*, & \text{otherwise}. 
		\end{cases}
		\end{equation}
		The structural rank of $T$, denoted by $\sprank(T)$, is the highest rank of all matrices with the same nonzero pattern as $S$.
	\end{defn}
	For an $L \times L$ structural matrix $T$, it is full structural rank if there are $L$ nonzero entries in $T$ that do not share common rows or columns \cite{Steffen2005Structural}. In other words,  there is a permutation of $S$ with all nonzero diagonal elements.
	
	To testify the dependency of the functions in $\mathscr{F}$ defined in \eqref{eq:Fset}, we consider $T_{\mathcal{C},\mathcal{R}}$ and check each submatrix of  $T_{\mathcal{C},\mathcal{R}}$ that has full structural rank.

	\begin{lem} \label{lem:struct}
		Consider any transfer matrix $T$ and the function set $\mathscr{F}$ in \eqref{eq:Fset}. The nonconstant rational functions in $\mathscr{F}$ are dependent, if there are two subsets
		$\bar{\mathcal{C}} \subseteq \mathcal{C}$ and 
		$\bar{\mathcal{R}} \subseteq \mathcal{R}$ such that
		\begin{equation} \label{eq:struct}
		\rank (T_{\bar{\mathcal{C}},\bar{\mathcal{R}}}) < \sprank (T_{\bar{\mathcal{C}},\bar{\mathcal{R}}}) = |\bar{\mathcal{C}}| = |\bar{\mathcal{R}}|.
		\end{equation}
	\end{lem}
	\begin{proof}
		The two equalities in \eqref{eq:struct} mean that $T_{\bar{\mathcal{C}},\bar{\mathcal{R}}}$ is a square matrix that has full structural rank. Therefore, the determinant $\det (T_{\bar{\mathcal{C}},\bar{\mathcal{R}}})$ is a function that is not consistently zero and can be written as combination of the functions in the set $\mathscr{F}$.
		While the inequality holds, i.e. $T_{\bar{\mathcal{C}},\bar{\mathcal{R}}}$ is not full rank, the determinant function $\det (T_{\bar{\mathcal{C}},\bar{\mathcal{R}}}) = 0$, from which a dependency of the functions  as components of $\det (T_{\bar{\mathcal{C}},\bar{\mathcal{R}}})$ is obtained. 
	\end{proof}
	
	We illustrate this lemma by using Example~\ref{ex:8nodes}. Consider a $3 \times 3$ matrix $T_{\bar{\mathcal{C}},\bar{\mathcal{R}}}$ as the submatrix of $T_{\mathcal{C},\mathcal{R}}$
%	 as
%	\begin{align*}
%	T_{\bar{\mathcal{C}},\bar{\mathcal{R}}} = \begin{bmatrix}
%	T_{61} & T_{62} & 0     \\
%	T_{71} & T_{72} & T_{73} \\
%	T_{81} & T_{82} & T_{83}
%	\end{bmatrix},
%	\end{align*}
	with $\bar{\mathcal{R}} = \{1, 2, 3\}$ and $\bar{\mathcal{C}} = \{6, 7, 8\}$. The structural matrix of $T_{\bar{\mathcal{C}},\bar{\mathcal{R}}}$ is given as
	\begin{equation*}
	S_{\bar{\mathcal{C}},\bar{\mathcal{R}}} = \begin{bmatrix}
	* & * & 0 \\
	* & * & * \\
	* & * & *
	\end{bmatrix},
	\end{equation*}
	implying that $\sprank(T_{\bar{\mathcal{C}},\bar{\mathcal{R}}}) = 3$. However, it follows from \eqref{ex:8nodes_det} that {$\det (T_{\bar{\mathcal{C}},\bar{\mathcal{R}}}) = T_{83} (T_{61} T_{72} - T_{62} T_{71}) - T_{73}(T_{61} T_{82} - T_{62} T_{81}) = 0$}, i.e. $\rank(T_{\bar{\mathcal{C}},\bar{\mathcal{R}}}) < 3$. Therefore, the nonconstant elements in $T_{\bar{\mathcal{C}},\bar{\mathcal{R}}}$ are dependent. 
	
	\begin{rem}
		To identify dependent functions in $\mathscr{F}$, we need to check the rank of each submatrix in $T_{\mathcal{C},\mathcal{R}}$, since the full rank of $T_{\mathcal{C},\mathcal{R}}$ does not guarantee the full rank of its submatrices. One example is $T_{\mathcal{C}, \mathcal{R}}$ in Example~\ref{ex:8nodes} that has full rank, while the submatrix $T_{\bar{\mathcal{C}},\bar{\mathcal{R}}}$ is rank deficient.
	\end{rem}
	
	To derive a tighter necessary condition for network identifiability than \eqref{eq:naive}, we proceed to an iterative elimination of the entries in $T_{\mathcal{C},\mathcal{R}}$, when dependency of the nonconstant element of $T_{\mathcal{C},\mathcal{R}}$ is found. A detailed scheme is described as follows.
	\begin{algorithm}
		\caption{Iterative Elimination of Dependent Functions}
		\begin{algorithmic}[1]
			\STATE \textbf{initialize} $\hat{T} = T_{\mathcal{C},\mathcal{R}}$, and the set $\mathscr{F}$ as defined in \eqref{eq:Fset}.
			
			\REPEAT
			\STATE Find $\bar{\mathcal{C}} \subseteq \mathcal{C}$ and 
			$\bar{\mathcal{R}} \subseteq \mathcal{R}$ such that 
			\begin{equation} \label{eq:rankdeg1}
			\rank (\hat{T}_{\bar{\mathcal{C}},\bar{\mathcal{R}}}) < \sprank (\hat{T}_{\bar{\mathcal{C}},\bar{\mathcal{R}}}) = |\bar{\mathcal{C}}| = |\bar{\mathcal{R}}|.
			\end{equation}
			
			\STATE Select a pair of $i \in \bar{\mathcal{R}}$,  $j \in \bar{\mathcal{C}}$  and remove the element $T_{ji}$ in $\mathscr{F}$, and let the element in $\hat{T}$ corresponding to ${T}_{ji}$ be zero.
			
			%		\IF{$|\mathscr{F}| < |\mathcal{E}|$ with $\mathcal{E}$ the edge set of the network}
			%		\STATE The network model set is not identifiable 
			%		\ENDIF
			\UNTIL There are no  $\bar{\mathcal{C}} \subseteq \mathcal{C}$ and 
			$\bar{\mathcal{R}} \subseteq \mathcal{R}$ satisfying \eqref{eq:rankdeg1}.
			
			\RETURN The reduced set $\hat{\mathscr{F}}$ associated with $\hat{T}$.		
		\end{algorithmic}
		\label{alg:entryremove}
	\end{algorithm}
	
	With the element elimination procedure in Algorithm~\ref{alg:entryremove}, $T_{\mathcal{C},\mathcal{R}}$ is sparsified as $\hat{T}$, in which nonconstant elements form a reduced set $\hat{\mathscr{F}}$ after removing the dependent rational functions in $\mathscr{F}$. Thereby, it then yields a new necessary condition for network identifiability as follows.
	\begin{thm}
		\label{thm:necessary}
		Consider the network model set $\mathcal{M}$ with $\mathcal{R}$ and $\mathcal{C}$ the sets of excited and measured vertices, respectively. Let $\hat{\mathscr{F}}$ be the set generated by Algorithm~\ref{alg:entryremove}. $\mathcal{M}$ is identifiable only if 
		$
		|\hat{\mathscr{F}}| \geq |\mathcal{E}|,
		$
		where $\mathcal{E}$ is the edge set of the network.
	\end{thm}
	\begin{proof}
		If $\mathcal{M}$ is identifiable, then it is clear that the number of independent functions in $\mathscr{F}$ should be greater than or equal to $|\mathcal{E}|$. Otherwise, there will be less number of equations than the number of unknown modules as indeterminate variables, such that the system of equations becomes underdetermined  and cannot yield a
		unique solution.
		
		From Algorithm~\ref{alg:entryremove}, the returned function set $\hat{\mathscr{F}}$ is generated by removing a subset of elements in $\mathscr{F}$ defined as in \eqref{eq:Fset}, and these removed elements are nonconstant rational functions that are dependent on the rest of  functions in $\hat{\mathscr{F}}$. Then, it is necessary to have $|\hat{\mathscr{F}}| \geq |\mathcal{E}|$ if $\mathcal{M}$ is identifiable.
	\end{proof}
	
	A special case is discussed, where there is a single excited vertex or only one measured vertex. In this case, we do not need to implement the elimination procedure in Algorithm~\ref{alg:entryremove}. 
	\begin{coro} \label{coro:RC1}
		Consider the network model set $\mathcal{M}$ with $\mathcal{R}$ and $\mathcal{C}$ the sets of excited and measured vertices, respectively. Let $|\mathcal{R}| = 1$ or $|\mathcal{C}| = 1$. If $\mathcal{M}$ is identifiable, then
%		\begin{equation} \label{eq:FE}	%		
			$|\mathscr{F}| \geq |\mathcal{E}|,$
%		\end{equation}
		where $\mathscr{F}$ is the set of nonconstant elements in $T_{\mathcal{C},\mathcal{R}}$.
	\end{coro}
	\begin{proof}
		If $|\mathcal{R}| = 1$, $T_{\mathcal{C},\mathcal{R}}$ is a matrix with only one column, and thus $	\rank (T_{\bar{\mathcal{C}},{\mathcal{R}}}) =  \sprank (T_{\bar{\mathcal{C}},{\mathcal{R}}})  = 1$, for any $\bar{\mathcal{C}} \subseteq \mathcal{C}$.
		Consequently, no dependent rational functions can be found in $\mathscr{F}$, and the result is immediate from Theorem~\ref{thm:necessary}. For the case that $|\mathcal{C}| = 1$, we can prove the statement in a similar way.
	\end{proof}
	
	\begin{exam}
		\label{ex:6nodes}
		In this example, we illustrate how to use Algorithm~\ref{alg:entryremove} and Theorem~\ref{thm:necessary} to check network identifiability. Consider a circular network in Fig.~\ref{fig:circle1} with six vertices, where $\mathcal{R} = \{1,2,3\}$ and $\mathcal{C} = \{4,5,6\}$. Initially, we have 
		\begin{equation*}
		\hat{T} = T_{\mathcal{C},\mathcal{R}} = \begin{bmatrix}
		T_{41} & T_{42} & T_{43} \\
		T_{51} & T_{52} & T_{53} \\
		T_{61} & T_{62} & T_{63} \\
		\end{bmatrix},
		\end{equation*}
		which gives $\mathscr{F} = \{T_{41}, T_{42}, T_{43}, T_{51}, T_{52}, T_{53}, T_{61}, T_{62}, T_{63}\}$. 
		Note that
		\begin{align*}
		T_{51} & = (1 - \phi_c)^{-1} G_{54} G_{43} G_{32} G_{21}, \\
		T_{41} & = (1 - \phi_c)^{-1} G_{43} G_{32} G_{21}, \\
		T_{52} & = (1 - \phi_c)^{-1} G_{54} G_{43} G_{32}, \\
		T_{42} & = (1 - \phi_c)^{-1} G_{43} G_{32},
		\end{align*}
%	$		T_{51}  = (1 - \phi_c)^{-1} G_{54} G_{43} G_{32} G_{21}, \
%	T_{41} = (1 - \phi_c)^{-1} G_{43} G_{32} G_{21}, \ 
%	T_{52}  = (1 - \phi_c)^{-1} G_{54} G_{43} G_{32},  \
%	T_{42}  = (1 - \phi_c)^{-1} G_{43} G_{32},$
		where $
		\phi_c : = G_{16} G_{65} G_{54} G_{43} G_{23} G_{21}
		$. The four equations satisfy
		$
		T_{51} T_{41}^{-1} = T_{52} T_{42}^{-1} = G_{54},
		$
		meaning that $T_{\{4,5\}, \{1,2\}}$ is rank deficient while $T_{\{4,5\}, \{1,2\}}$ has full structural rank and thus fulfils \eqref{eq:rankdeg1}. We select $T_{41}$ and remove this element from $\mathscr{F}$ and let 
		\begin{equation}
		\hat{T} = \begin{bmatrix}
		0 &  T_{42} & T_{43} \\
		T_{51} & T_{52} & T_{53} \\
		T_{61} & T_{62} & T_{63} \\
		\end{bmatrix}.
		\end{equation}
		
		Further, we observe that 
		$
		T_{52} T_{42}^{-1} = T_{53} T_{43}^{-1} = G_{54},
		$
		such that $T_{\{4,5\}, \{2,3\}}$ satisfies \eqref{eq:rankdeg1}. We can further eliminate an element $T_{42}$ in $\mathscr{F}$ and let $T_{42} = 0$. We can repeat this process and eventually obtain a reduced set $\hat{\mathscr{F}} = \{T_{43}, T_{53}, T_{61}, T_{62}, T_{63}\}$ associated with the matrix
		\begin{equation*}
		\hat{T} = \begin{bmatrix}
		0 & 0 & T_{43} \\
		0 & 0 & T_{53} \\
		T_{61} & T_{62} & T_{63} \\
		\end{bmatrix}.
		\end{equation*}
		Since there are only $5$ elements, which is less than the number of unknown modules in the original network in Fig.~\ref{fig:circle1}, it follows from Theorem~\ref{thm:necessary} that the model set is not identifiable.
		
		\begin{figure}[!tp] 
			\centering
			\includegraphics[width=0.2\textwidth]{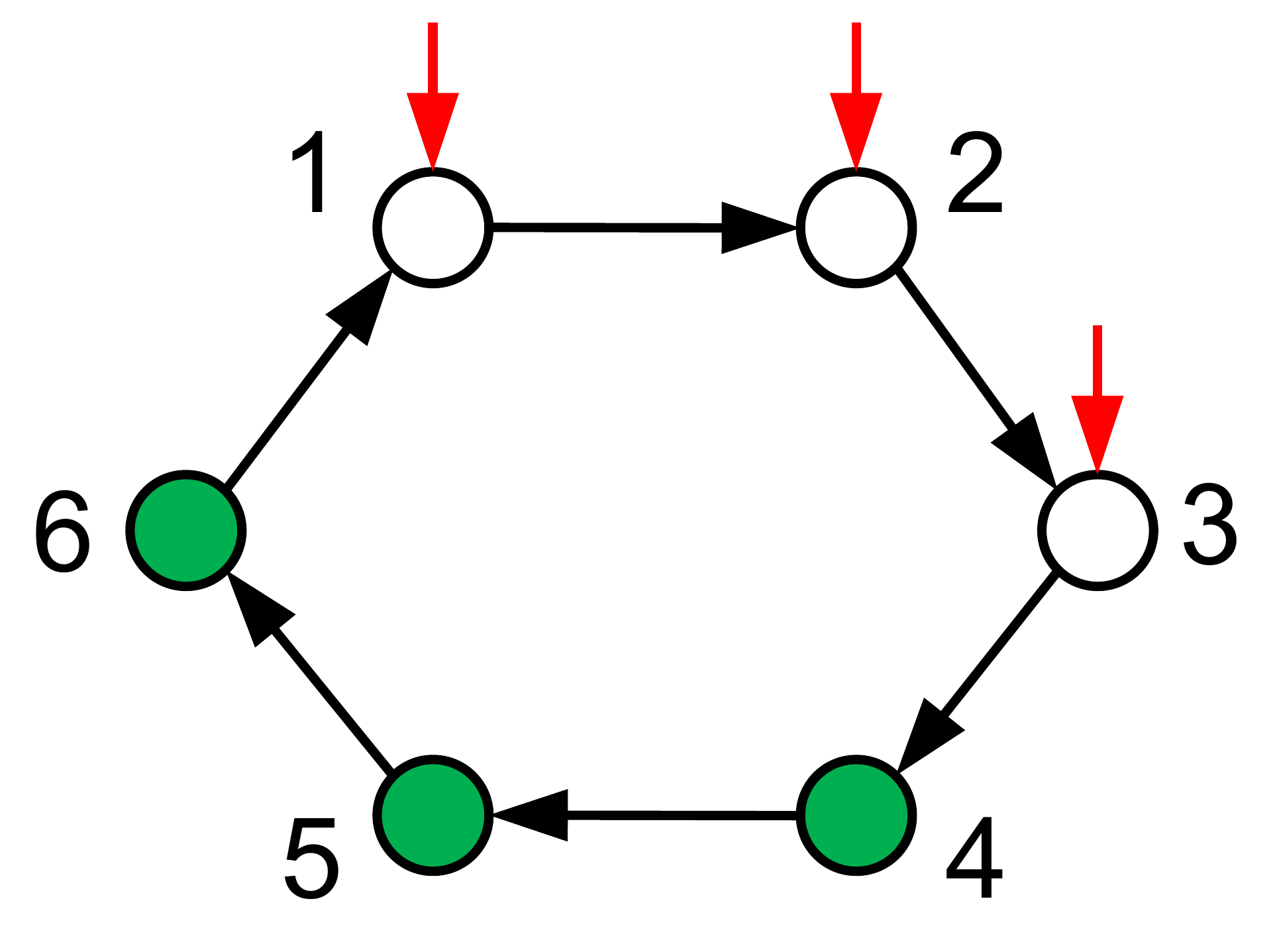}
			\caption{A six-vertex dynamic network with $\mathcal{R} = \{1, 2, 3\}$ and $\mathcal{C}= \{4, 5, 6\}$. The network model set is not identifiable by applying Theorem~\ref{thm:necessary}.}
			\label{fig:circle1}
		\end{figure}
	\end{exam}

	Although Algorithm~\ref{alg:entryremove} and Theorem~\ref{thm:necessary} provide a necessary condition for network identifiability in the partial excitation and partial measurement setting, this  condition is limited since it is difficult to check 
	the relevant ranks in \eqref{eq:rankdeg1} when encountering a large-scale network. Therefore, one of the main contributions of this paper will be to provide a graphical characterisation for dependent functions in $\mathcal{F}$, and a graph-based condition will be given in the next section.

	\subsection{A Graph-Based Condition}

	In this section, we provide a graph-based version for the condition in Theorem~\ref{thm:necessary}. The basic idea is to use graphical alternatives to characterise the rank and the structural rank appearing in \eqref{eq:rankdeg1}.

	First, a graph-based characterisation of a transfer matrix is revisited. Consider the transfer matrix $T_{\mathcal{C},\mathcal{R}}$ in which each nonconstant element is a function of $\theta \in \Theta$. It then follows from e.g., \cite{vanderWoude1991graph} that 
	\begin{equation} \label{eq:vdp_grank}
	\max_{\theta \in \Theta}\rank (T_{\mathcal{C},\mathcal{R}}(q, \theta)) = b_{\mathcal{R}\rightarrow \mathcal{C}},
	\end{equation}
	where $b_{\mathcal{R}\rightarrow \mathcal{C}}$ denotes the maximal number of vertex-disjoint paths from $\mathcal{R}$ to $\mathcal{C}$. The expression on the left-hand side of \eqref{eq:vdp_grank} is referred to as the \textit{generic rank} of the matrix $T_{\mathcal{C},\mathcal{R}}$. From \eqref{eq:vdp_grank}, we further have
	\begin{equation} \label{eq:vdp_rank}
	\rank (T_{\mathcal{C},\mathcal{R}}(q,\theta)) \leq b_{\mathcal{R}\rightarrow \mathcal{C}},
	\end{equation}
	in which the \textit{equality generically holds}, i.e. holds for almost all $\theta \in \Theta$.

	Next, the structural rank of $T_{\mathcal{C},\mathcal{R}}$ is interpreted graphically. The structural matrix of $T_{\mathcal{C},\mathcal{R}}$ that represents its nonzero pattern is associated with a so-called \textit{bipartite graph} defined by a triplet $\mathcal{B} : = (\mathcal{R},\mathcal{C},\mathcal{E}_b)$, where  every edge $(i,j) \in \mathcal{E}_b$ links a vertex $i \in \mathcal{R}$ and $j \in \mathcal{C}$, if $T_{ji} \neq 0$ (see e.g., \cite{godsil2013algebraic}).  
	 Note that we allow a vertex $i$ to be  excited and measured simultaneously, thus an edge $(i,i) \in \mathcal{E}_b$ may exist. 
	\begin{defn}
			In a bipartite graph $ \mathcal{B}= (\mathcal{R},\mathcal{C},\mathcal{E}_b)$, a \textbf{matching} between two sets $\bar{\mathcal{R}} \subseteq \mathcal{R}$ and $\bar{\mathcal{C}} \subseteq \mathcal{C}$ is a set of pairwise edges between $\bar{\mathcal{R}}$ and $\bar{\mathcal{C}}$ that do not share any common vertices. Furthermore, a \textbf{maximum(-cardinality) matching} between $\bar{\mathcal{R}}$ and $\bar{\mathcal{C}}$, denoted by $\mathscr{M}(\bar{\mathcal{R}}, \bar{\mathcal{C}})$, is a matching between $\bar{\mathcal{R}}$ and $\bar{\mathcal{C}}$ with the largest possible number of edges.
	\end{defn}
 	This concept then leads to the following result. 
	\begin{lem} \label{lem:maxmatch}  \cite{godsil2013algebraic}
		Consider the matrix $T$ in \eqref{eq:T}, and any $\bar{\mathcal{R}} \subseteq \mathcal{R}$, $\bar{\mathcal{C}} \subseteq \mathcal{C}$. It holds that
		$
		\sprank (T_{\bar{\mathcal{C}},\bar{\mathcal{R}}})
		=
		|\mathscr{M}(\bar{\mathcal{R}}, \bar{\mathcal{C}})|,
		$
		where $\mathscr{M}(\bar{\mathcal{R}}, \bar{\mathcal{C}})$ is any maximum matching between $\bar{\mathcal{R}}$, $\bar{\mathcal{C}}$ in the associated bipartite graph.
	\end{lem}
	The following example is used to demonstrate how to determine the structural rank of a matrix through the maximum matching in its associated bipartite graph. 
	\begin{exam}
		\label{ex:4nodes_bp}
		A four-vertex dynamic network is shown in Fig.~\ref{fig:fournodes2}, in which $\mathcal{R} = \{1, 2, 4\}$, and $\mathcal{C} = \{2, 3, 4\}$. From Definition~\ref{defn:strucmat}, this network has the structural matrix as
		\begin{equation*}
		S = \begin{bmatrix}
		* & 0  & 0 & 0\\
		* & *  & 0 & 0\\
		* & * & * & *\\
		* & * & * & *\\
		\end{bmatrix}, 
		\ \text{and} \
		S_{\mathcal{C},\mathcal{R}} = \begin{bmatrix}
		* & * & 0\\
		* & * & *\\
		* & * & *\\
		\end{bmatrix},
		\end{equation*} 
		The associated bipartite graph $\mathcal{B}$ of $T_{\mathcal{C},\mathcal{R}}$ is constructed as in Fig.~\ref{fig:bipartite}. 
		%	 There is edge $(4,4)$ in $\mathcal{B}$ since vertex $4$ is on a directed cycle. Every other edge $(i,j)$ in $\mathcal{B}$ corresponds to a directed path from a excited vertex $i$ to a measured vertex $j$.
		A maximum matching of this bipartite graph is given as $\{(1,2), (2,3), (4,4)\}$, which has cardinality 3. Thus, $\sprank(T_{\mathcal{C},\mathcal{R}}) = 3$.
		Note that the maximum matchings of a bipartite graph may not be unique. An alternative in this case can be $\{(1,2), (2,4), (4,3)\}$.
		\begin{figure}[!tp] 
			\begin{minipage}[t]{0.5\linewidth}
				\centering
				\includegraphics[width=0.8\textwidth]{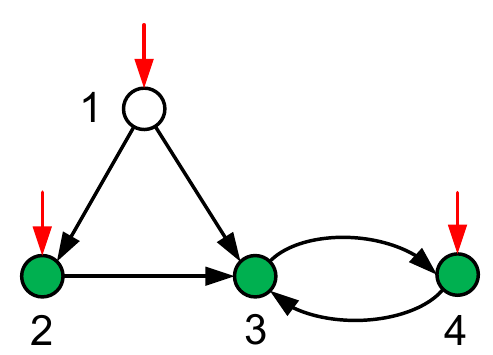}
				\subcaption{}
				\label{fig:fournodes2}
			\end{minipage}%
			\begin{minipage}[t]{0.5\linewidth}
				\centering
				\includegraphics[width=0.7\textwidth]{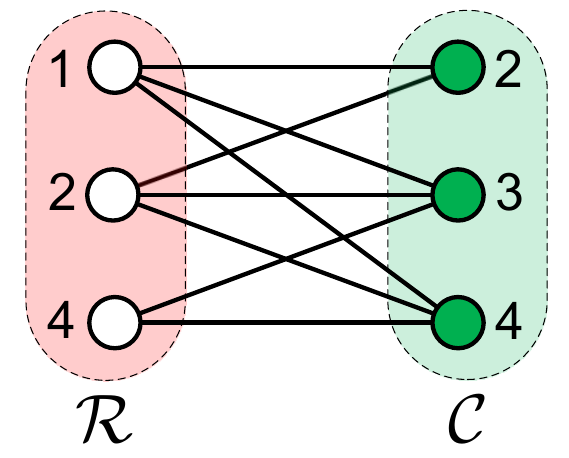}
				\subcaption{}
				\label{fig:bipartite}
			\end{minipage}
			\caption{(a) A four-vertex dynamic network with $\mathcal{R} = \{1, 2, 4\}$ and $\mathcal{C}= \{2, 3, 4\}$. (b) The associated bipartite graph $\mathcal{B}$ with a maximum matching $\{(1,2), (2,3), (4,4)\}$.}
		\end{figure}
	\end{exam}
	
	By means of Lemma~\ref{lem:maxmatch} and the relation \eqref{eq:vdp_rank}, the result in Lemma~\ref{lem:struct} can be reformulated on the basis of graphs.
	\begin{coro}
		\label{coro:struct}
		Consider the network model in \eqref{eq:net} with the underlying graph $\mathcal{G}$. Let $\mathcal{R}$ and $\mathcal{C}$ be the sets of vertices that are excited and measured, respectively. The nonconstant rational functions in the set $\mathscr{F}$ defined in \eqref{eq:Fset} are dependent, if there exist two subsets
		$\bar{\mathcal{C}} \subseteq \mathcal{C}$ and 
		$\bar{\mathcal{R}} \subseteq \mathcal{R}$ such that
		\begin{equation} \label{eq:struct1}
		b_{\bar{\mathcal{R}} \rightarrow \bar{\mathcal{C}}}  < |\mathscr{M} (\bar{\mathcal{R}}, \bar{\mathcal{C}})| = |\bar{\mathcal{C}}| = |\bar{\mathcal{R}}|,
		\end{equation}
		where $\mathscr{M}(\bar{\mathcal{R}},\bar{\mathcal{C}})$ is a maximum matching between $\bar{\mathcal{R}}$ and $\bar{\mathcal{C}}$ in the bipartite graph $\mathcal{B} := \{\mathcal{R},\mathcal{C},\mathcal{E}_b\}$, and $b_{\bar{\mathcal{R}} \rightarrow \bar{\mathcal{C}}}$ is the maximum number of vertex-disjoint paths from $\bar{\mathcal{R}}$ to $\bar{\mathcal{C}}$ in $\mathcal{G}$. 
	\end{coro}

	We show how to apply Corollary~\ref{coro:struct} to check if the nonconstant elements in $\mathscr{F}$ are dependent in Example~\ref{ex:4nodes_bp}. Consider two subsets $\bar{\mathcal{R}} = \{1,2\}$ and $\bar{\mathcal{C}} = \{3,4\}$. Observe that $\{(1,3), (2,4)\}$ is a maximum matching between the two subsets with $|\mathscr{M} (\bar{\mathcal{R}}, \bar{\mathcal{C}})|  = 2$, while the maximum number of vertex disjoint paths from $\bar{\mathcal{R}}$ to $\bar{\mathcal{C}}$ is only 1. Therefore, the elements in $\mathscr{F}$ are not independent.
	
	Now we derive a necessary condition for network identifiability based on a more comprehensive procedure that iteratively removes edges in the bipartite graph $\mathcal{B}$ of $T_{\mathcal{C},\mathcal{R}}$.
	Note that each edge in  $\mathcal{B}$ is associated with an nonconstant entry in $T_{\mathcal{C},\mathcal{R}}$, or an element in $\mathscr{F}$. This simplification process corresponds to the element removal steps in Algorithm~\ref{alg:entryremove}.
	 Consider a network model as in \eqref{eq:net} with $\mathcal{R}$ and $\mathcal{C}$ the excited and measured vertices, where $|\mathcal{R}| \geq 2$ and $|\mathcal{C}| \geq 2$. In the case where $|\mathcal{R}| = 1$ or $|\mathcal{C}| = 1$, we can simply apply Corollary~\ref{coro:RC1}.
	Let $\mathcal{B} : = (\mathcal{R},\mathcal{C},\mathcal{E}_b)$ be the bipartite graph associated with $T_{\mathcal{C},\mathcal{R}}$. A graph simplification process is performed on $\mathcal{B}$, see Algorithm~\ref{alg:edgeremove}, in which the set $\mathcal{E}_b (\bar{\mathcal{R}}, \bar{\mathcal{C}}) \subseteq \mathcal{E}_b ({\mathcal{R}}, {\mathcal{C}})$ collects all the edges between 
	$\bar{\mathcal{R}} \subseteq \mathcal{R}$ and $\bar{\mathcal{C}} \subseteq \mathcal{C}$.

	\begin{algorithm}
		\caption{Edge-Removal in Bipartite Graph}
		\begin{algorithmic}[1]
			\STATE  \textbf{initialize} $\hat{\mathcal{E}}_t = \emptyset$
			
			\FOR{$k = 2 : \min\{|\mathcal{R}|, |\mathcal{C}|\}$}
	
			\STATE For all $\bar{\mathcal{C}} \subseteq \mathcal{C}$ and 
			$\bar{\mathcal{R}} \subseteq \mathcal{R}$ with $|\bar{\mathcal{C}}| = 
			|\bar{\mathcal{R}}| = k$,
	
			\IF{$b_{\bar{\mathcal{R}} \rightarrow \bar{\mathcal{C}}} < |\mathscr{M} (\bar{\mathcal{R}}, \bar{\mathcal{C}})| = k$ and $\mathcal{E}_b (\bar{\mathcal{R}}, \bar{\mathcal{C}}) \nsubseteq \hat{\mathcal{E}}_t $}
			\STATE $\hat{\mathcal{E}}_t \leftarrow \mathcal{E}_b (\bar{\mathcal{R}}, \bar{\mathcal{C}}) \cup \hat{\mathcal{E}}_t$
			\STATE Remove an arbitrary edge $(i,j)$ in the bipartite graph with $i \in \bar{\mathcal{R}}$, $j \in \bar{\mathcal{C}}$.
			\ENDIF
			
			\ENDFOR 
			
			\RETURN A simplified bipartite graph  $\hat{\mathcal{B}}$.	
			%		\IF{$|\mathcal{E}_b| < |\mathscr{V}|$ with $\mathscr{V}$ defined in \eqref{eq:variables},}
			%			\RETURN $\mathcal{M}$ is not identifiable.
			%		\ENDIF
		\end{algorithmic}
		\label{alg:edgeremove}
	\end{algorithm}
	
	The path-based  characterisation \eqref{eq:vdp_rank} for the rank of the transfer matrix $T_{\mathcal{C},\mathcal{R}}$ does not hold for a matrix $\hat{T}$ in Algorithm~\ref{alg:entryremove} that is a sparsification of $T_{\mathcal{C},\mathcal{R}}$ with certain entries in $T_{\mathcal{C},\mathcal{R}}$ assigned to zero. Therefore, in Algorithm~\ref{alg:edgeremove}, we identify the independent functions in $\mathscr{F}$ based on the original $T_{\mathcal{C},\mathcal{R}}$ matrix in order to utilize the path-based characterisation \eqref{eq:vdp_rank}. To this end, we define an edge set $\hat{\mathcal{E}}_t$, which collects all the edges corresponding to the rational functions whose dependency have been exploited. If the subsets $\bar{\mathcal{C}} \subseteq \mathcal{C}$ and 
	$\bar{\mathcal{R}} \subseteq \mathcal{R}$ have been found to satisfy $b_{\bar{\mathcal{R}} \rightarrow \bar{\mathcal{C}}} < |\mathscr{M}(\bar{\mathcal{R}}, \bar{\mathcal{C}})|$, {meaning that the matrix $T_{\bar{\mathcal{C}},\bar{\mathcal{R}}}$ contains dependent elements}, we then include all the edges between $\bar{\mathcal{R}}$ and $\bar{\mathcal{C}}$ in $\mathcal{B}$ into the set $\hat{\mathcal{E}}_t$. For the latter iteration, it is required that $\mathcal{E}_b (\bar{\mathcal{R}}, \bar{\mathcal{C}}) \nsubseteq \hat{\mathcal{E}}_t $, namely, we avoid checking subsets $\bar{\mathcal{C}}$ and 
	$\bar{\mathcal{R}}$ if all the edges between the two subsets are in $\hat{\mathcal{E}}_t$, as the dependency of the relevant functions associated with $\hat{\mathcal{E}}_t$ has been examined. We will illustrate the procedure of Algorithm~\ref{alg:edgeremove} in Example~\ref{ex:bpremove}.

	Algorithm~\ref{alg:edgeremove} starts with inspecting $2 \times 2$ submatrices $T_{\bar{\mathcal{C}},\bar{\mathcal{R}}}$ in $T_{\mathcal{C},\mathcal{R}}$ 
	which are the smallest submatrices for detecting dependent elements in $\mathscr{F}$.
%	of full structural rank, i.e. $|\mathscr{M} (\bar{\mathcal{R}}, \bar{\mathcal{C}})| = 2$. If $b_{\bar{\mathcal{R}} \rightarrow \bar{\mathcal{C}}} < 2$ holds in the graph $\mathcal{G}$, meaning that $T_{\bar{\mathcal{C}},\bar{\mathcal{R}}}$ is not full rank, then the four elements in $T_{\bar{\mathcal{C}},\bar{\mathcal{R}}}$ are dependent and thus one of them can be eliminated in $\mathscr{F}$ that corresponds to the removal of the corresponding edge in the bipartite graph $\mathcal{B}$. 
	Note that we do not start with the higher-dimensional submatrices in $T_{\mathcal{C},\mathcal{R}}$, since it may not be able to find dependent elements in some lower-dimensional submatrices. Take the network in Fig.~\ref{fig:circle1} as an example, if we first check the submatrix $T_{\bar{\mathcal{C}},\bar{\mathcal{R}}}$ with $\bar{\mathcal{R}} = \{1,2,3\}$ and $\bar{\mathcal{R}} = \{4,5,6\}$, then all the edges of the associated bipartite graph will be included in $\hat{\mathcal{E}}_t$. As a result, Algorithm~\ref{alg:edgeremove} will stop checking dependent elements in $2 \times 2$ submatrices, which leads to a more conservative result.
	
	With the simplified bipartite graph $\hat{\mathcal{B}}$ generated by Algorithm~\ref{alg:edgeremove}, we obtain a graph-based necessary condition for network identifiability as follows.
	\begin{coro}
		\label{coro:necessary_graph}
		Consider the network model set $\mathcal{M}$ with $\mathcal{R}$ and $\mathcal{C}$ the sets of excited and measured vertices, respectively. Let $\hat{\mathcal{E}_b}$ be the edge set of the simplified bipartite graph $\hat{\mathcal{B}}$ generated by Algorithm~\ref{alg:edgeremove}. The model set $\mathcal{M}$ is identifiable only if $|\hat{\mathcal{E}_b}| \geq |\mathcal{E}|$.
	\end{coro}

	\begin{exam}
		\label{ex:bpremove}
		We implement Corollary~\ref{coro:necessary_graph} to check identifiability of the model set of the network in Fig.~\ref{fig:circle1}. The associated bipartite graph $\mathcal{B}$ is shown in Fig.~\ref{fig:circlebip}. Consider the subsets $\bar{\mathcal{R}} = \{1, 2\}$, $\bar{\mathcal{C}} = \{4, 5\}$ with $|\mathscr{M} (\bar{\mathcal{R}}, \bar{\mathcal{C}})| = 2$. We find that the maximal number of disjoint paths from $\bar{\mathcal{R}}$ to $\bar{\mathcal{C}}$ in Fig.~\ref{fig:circle1} is less than $2$, implying that the four entries in $T_{\bar{\mathcal{C}},\bar{\mathcal{R}}}$ are dependent. Therefore, we remove an edge $(1,4)$ in $\mathcal{B}$ and meanwhile obtain the edge set $\hat{\mathcal{E}}_t = \{(1,4), (2,4), (1,5), (2,5)\}$. 
		
		We proceed to  check the subsets $\bar{\mathcal{R}} = \{1, 2\}$, $\bar{\mathcal{C}} = \{5, 6\}$, which allows to remove an edge $(1, 5)$ in $\mathcal{B}$, and the edge set $\hat{\mathcal{E}}_t$ is enlarged as $\hat{\mathcal{E}}_t = \{(1,4), (2,4), (1,5), (2,5), (1,6), (2, 6)\}$. This process can be continued by considering $\bar{\mathcal{R}} = \{2, 3\}$, $\bar{\mathcal{C}} = \{4, 5\}$ and then $\bar{\mathcal{R}} = \{2, 3\}$, $\bar{\mathcal{C}} = \{5, 6\}$. Accordingly, we remove two more edges $(2,4)$ and $(2,5)$ in $\mathcal{B}$, yielding a reduced bipartite graph with five edges, see the edges in Fig.~\ref{fig:circlebip} indicated by solid lines. Note that although there are more pairs satisfying \eqref{eq:struct1}, e.g., $\bar{\mathcal{R}} = \{1, 3\}$, $\bar{\mathcal{C}} = \{4, 6\}$ and $\bar{\mathcal{R}} = \{1, 2, 3\}$, $\bar{\mathcal{C}} = \{4, 5, 6\}$, we no longer impose an edge removal step, since $\hat{\mathcal{E}}_t$ now has covered all the edges in $\mathcal{B}$. Eventually, it follows from Corollary~\ref{coro:necessary_graph} that the model set of the network in Fig.~\ref{fig:circle1} is not identifiable as
		$|\hat{\mathcal{E}_b}| = 5 < |\mathcal{E}| = 6$.
		\begin{figure}[!tp] 
			\centering
			\includegraphics[width=0.18\textwidth]{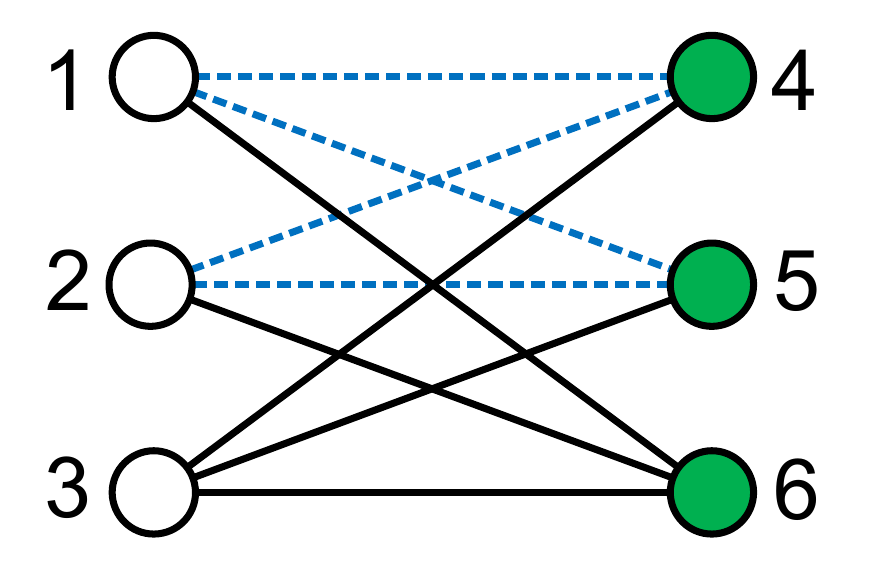}
			\caption{The bipartite graph associated with the network in Fig.~\ref{fig:circle1}. The dashed lines represent the edges can be removed.}
			\label{fig:circlebip}
		\end{figure}
	\end{exam}

	\subsection{Identifiability of Circular Networks}
	
	In this section, we zoom in network identifiability of directed circular  networks, 	
	which also includes isolated cycles in a larger network, namely, the ones that do not share any common vertices with other cycles. The result here appears as the very first necessary and sufficient condition for identifiability of cyclic graphs. In the existing work \cite{Bazanella2019CDC}, only sufficient conditions for identifiability of circular networks are presented:
	\begin{enumerate}
		\item A directed circular network is identifiable if $\mathcal{R} \cup \mathcal{C} = \mathcal{V}$ and $\mathcal{R} \cap \mathcal{C} \ne \emptyset$.
		\item In the special case that the vertex number $L$ is even and larger than 3, a  circular network is identifiable if its nodes are alternately measured and excited.
	\end{enumerate}
	
	However, the above two conditions can be conservative in some circumstances, see Example~\ref{ex:circle}, in which the network does not satisfy the two conditions. However this network is actually identifiable.  
    We will show this by exploring a new necessary and sufficient condition for identifiability in circular networks.
	\begin{thm} \label{thm:circle}
		A directed circular network is identifiable if and only if $\mathcal{R} \cup \mathcal{C} = \mathcal{V}$ and one of the following conditions holds.
		\begin{enumerate}
			\item $|\mathcal{R}| = 1$, $\mathcal{R} \subset \mathcal{C}$;
			
			\item $|\mathcal{C}| = 1$, $\mathcal{C} \subset \mathcal{R}$;
			
			\item  $|\mathcal{R}| \geq 2$, $|\mathcal{C}| \geq 2$, and 
			there are at least two vertex disjoint paths from $\mathcal{R}$ to $\mathcal{C}$ in the cycle.
		\end{enumerate}
	\end{thm}
	\begin{proof}
		Consider a circular network with $L$ vertices, which also has $L$ unknown modules.
		
		\textit{Necessity}: It has been shown in \cite{Bazanella2019CDC} that $\mathcal{R} \cap \mathcal{C} = \mathcal{V}$ is a necessary condition for identifiability of any directed network. In the following, we prove the three conditions by contradiction.
		In the case that $|\mathcal{R}| = 1$, i.e., there is only a single excited vertex, we assume this vertex to be unmeasured. It means $|\mathcal{C}| = L - 1$, and there are only $L-1$ nonconstant entries in $T_{\mathcal{C},\mathcal{R}}$, from which, it is impossible to recover $L$ unknown modules. This leads to a contradiction to identifiability. The proof for the case that $|\mathcal{C}| = 1$ follows similarly and thus is omitted here. 
		
		Next, we show that the existence of two vertex disjoint paths from $\mathcal{R}$ to $\mathcal{C}$ is necessary for network identifiability of a directed cycle when $|\mathcal{R}| = K \geq 2$, $|\mathcal{C}| = N \geq 2$. Let $\mathcal{R} = \{j_1, ... ,j_K \}$ and $\mathcal{C} = \{i_1, ... ,i_N \}$. In this case,  $\hat{T} = T_{\mathcal{C},\mathcal{R}}$ is an $N \times K$ transfer matrix, where each entry is nonconstant (see Example~\ref{ex:6nodes}). Suppose that there does not exist two or more vertex disjoint paths in the cycle. Then, $\mathcal{C} \cap \mathcal{R} = \emptyset$, i.e. $N + K = L$. Furthermore, each $2 \times 2$ submatrix of  $\hat{T}$, e.g.,
		\begin{equation*}
		T_{\{i_1,i_2\}, \{j_1,j_2\}} = \begin{bmatrix}
		T_{i_1, j_1} & T_{i_1,j_2} \\
		T_{i_2, j_1} & T_{i_2, j_2}
		\end{bmatrix},
		\end{equation*}
		is not full rank. Following the element removal process in Algorithm~\ref{alg:entryremove}, we can assign $T_{i_1, j_1} = 0$ in $T$. Repeating the operation for each $2 \times 2$ submatrix of $T$, it leads to 
		\begin{equation*}
		\hat{T} = \begin{bmatrix}
		0 & \cdots & 0 & T_{i_1, j_K} \\
		\vdots & \ddots & \vdots & \vdots \\
		0 & \cdots & 0 & T_{i_{N-1}, j_K} \\
		T_{i_N, j_1} & \cdots & T_{i_{N}, j_{K-1}} & T_{i_N, j_K}
		\end{bmatrix},
		\end{equation*}
		whose bipartite graph contains only $N + K -1$ edges. It follows from Theorem~\ref{thm:necessary} that the circular network is not identifiable. Therefore, condition 3) is necessary.
		
		\textit{Sufficiency}: It directly obtain from \cite{Bazanella2019CDC} that the conditions 1) and 2) are sufficient for network identifiability of the circular network. Now we focus on the case that there are more than one excited and measured vertices. Without loss of generality, let vertices $1$ and $i$ be excited and vertices $j$, $k$ be measured, with $1 \leq k < i \leq j \leq L$. Moreover, the two paths from vertices $j_1$ to $i_1$ and from vertices $j_2$ to $i_2$ are vertex disjoint. Denote $
		\phi_c = G_{12} G_{23} \cdots G_{L-1,{L}} G_{L1}
		$ as the transfer function for the cycle that starts and ends at vertex $1$. Note that 
		\begin{align*}
		T_{k1} &= (1 - \phi_c)^{-1} G_{12}  \cdots G_{k-1,k}, \\
		T_{j1} &= (1 - \phi_c)^{-1} G_{12}  \cdots G_{k-1,k} G_{k,k+1} \cdots G_{j-1,j}, \\
		T_{ki} &= (1 - \phi_c)^{-1} G_{i,i+1}  \cdots G_{j-1,j} G_{j,j+1}   \cdots G_{k-1,k}, \\
		T_{ji} &= (1 - \phi_c)^{-1} G_{i,i+1}  \cdots G_{j-1,j}.
		\end{align*}  
		Therefore, we obtain
		\begin{equation} \label{eq:phic}
		T_{j1} T_{k1}^{-1} T_{ki} T_{ji}^{-1} = G_{k,k+1} \cdots G_{j-1,j} G_{j,j+1}   \cdots G_{k-1,k},
		\end{equation}
		which is equivalent to the cycle transfer function $\phi_c$. 
		
		For any vertex $u$ on the cycle, identifiability of $G_{u,u+1}$ can be analyzed as follows. If both vertices $u$  and $u+1$ are measured, then we can find an excited vertex $s$ such that 
		$G_{u,u+1} = T_{u+1,s} T_{u,s}$; If vertex $u$ is excited while $u+1$ is measured, then we have $G_{u,u+1} = (1 - \phi_c) T_{u,u+1}$, where $\phi_c$ is identified from \eqref{eq:phic}. If $u, u+1 \in \mathcal{R}$ or $u \in \mathcal{C}, u+1 \in \mathcal{R}$, we can prove identifiability of $G_{u,u+1}$ in a similar way. As the vertex $u$ can be chosen arbitrarily in the cycle, we thus can identify all the modules on the cycle.
	\end{proof}
	
	The first two conditions in Theorem~\ref{thm:circle} can be interpreted as that a vertex has to be measured (excited) if it is the only excited (measured) vertex in the cycle. 
	The minimal number of signals, i.e., $\gamma: = |\mathcal{R}| + |\mathcal{C}|$, required for identifiability of a circular network is now discussed. When the number of vertices $L \leq 3$ in the cycle, we have $\gamma = L + 1$, since at least one vertex has to be excited and measured simultaneously. When $L > 3$, then identifiability is guaranteed if $\gamma = L$ with two vertex disjoint paths from the excited vertices to the measured ones. Note that it is not necessary to impose the vertices to be alternately measured and excited as in \cite{Bazanella2019CDC}. In the following example, identifiability of the circular network model set cannot be determined by the conditions in \cite{Bazanella2019CDC}, while it can be checked by using Theorem~\ref{thm:circle}. 
	\begin{exam}
		\label{ex:circle}
		Consider a circular network in Fig.~\ref{fig:circle} with six vertices with $\mathcal{R} = \{1,3,4\}$ and $\mathcal{C} = \{2,5,6\}$. Observe that this network does not satisfy the identifiability conditions proposed in \cite{Bazanella2019CDC}. However,  there are two vertex disjoint paths from $\mathcal{R}$ to $\mathcal{C}$, highlighted by the blue edges. 
		Therefore, the third condition in Theorem~\ref{thm:circle} is fulfilled, showing identifiability of the model set of this circular network. 
		
		\begin{figure}[!tp]
			\centering
			\includegraphics[width=0.2\textwidth]{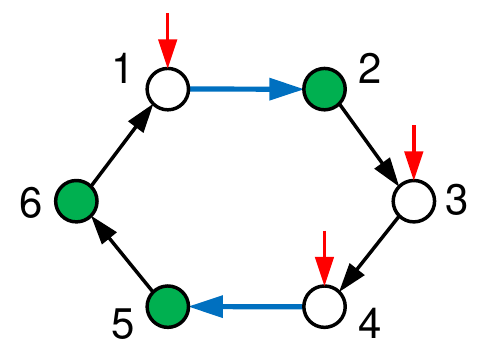}
			\caption{A circular  network with two vertex disjoint paths indicated by the blue edges. The model set is identifiable.}
			\label{fig:circle}
		\end{figure}%
	\end{exam}

	\section{Conclusion}
	\label{sec:conclusion}
	
	In this paper, we have analysed identifiability of a dynamic network where only partial excitation and measurement signals are available.  We presented a necessary condition for identifiability of general networks, where identifiability is determined by the dependency of a set of rational functions with the parametrised modules as indeterminate variables. The merit of this result is that the necessary condition can be reinterpreted as a graph-theoretical condition only dependent on network topology. Moreover, we obtain a necessary and sufficient identifiability condition for circular  networks.

	\bibliographystyle{IEEEtran}        
	\bibliography{netid}

\end{document}